\documentclass[11pt, a4paper,svgnames]{article} 
\usepackage[margin=1.206in]{geometry}
\usepackage[english]{babel}
\usepackage[latin1]{inputenc}
\usepackage{mathtext}
\usepackage{amsmath, amssymb}
\usepackage{graphicx}
\usepackage{caption}
\usepackage{ mathrsfs }
\usepackage{amsthm}
\usepackage[makeroom]{cancel}
\usepackage{ tipa }
\usepackage[numbers,square]{natbib}
\usepackage{titlesec}
\usepackage{caption}
\usepackage{subcaption}
\usepackage{dsfont}
\usepackage{xcolor}
\usepackage{url}

\usepackage[breaklinks,colorlinks=true,linkcolor=Blue,citecolor=Green,urlcolor=Blue]{hyperref}

\usepackage{tikz}
\usetikzlibrary{shadings}
\usetikzlibrary{fadings}
\usepackage{pgfplots}
\usetikzlibrary{pgfplots.groupplots}
\pgfplotsset{compat=1.9}
\usepackage{pgfplotstable}
\usepackage{subcaption}
\usepackage[nameinlink]{cleveref}

\usepackage{tkz-graph}
\usepackage{float}

\tikzstyle{outerframe}=[very thick,Green,fill=Green!30!white,opacity=0.5,even odd rule]
\tikzstyle{innerframe}=[very thick,Blue,fill=Blue!30!white,opacity=0.5,even odd rule]

\newcommand{\rank}{\operatorname{rank}}
\newcommand{\tr}{\operatorname{tr}}

\newcommand{\sgn}{\operatorname{sgn}}

\newcommand{\Norm}[2]{ \left\lVert #1 \right\rVert_{#2}}

\newcommand{\RE}{\operatorname{Re}}

\newcommand{\diag}{\operatorname{diag}}
\newcommand{\MOD}{\operatorname{mod}}
\newcommand{\hh}{\mathcal{H}}

\newcommand{\est}[1]{\hat{#1}}
\newcommand{\tx}[1]{{#1}}
\newcommand{\Tor}{\mathbb{T}}

\newcommand{\Co}{\mathbb{C}}

\newtheorem{theorem}{Theorem}[section]
\newtheorem{lemma}[theorem]{Lemma}

\newtheorem{corollary}[theorem]{Corollary}

\makeatletter
\newcommand{\footremember}[2]{%
\footnote{#2}
\newcounter{#1}
\setcounter{#1}{\value{footnote}}%
}
\newcommand{\footrecall}[1]{%
\footnotemark[\value{#1}]%
}
\makeatother
\title{\Large\bfseries On recovery guarantees for angular synchronization}
\date{\today}
\author{Frank Filbir\footremember{HMGU}{Mathematical Imaging and Data Analysis, Department of Scientific Computing, Helmholtz Center Munich, 85764 Neuherberg, Germany 
(\href{mailto:filbir@helmholtz-muenchen.de}{filbir@helmholtz-muenchen.de})} \and Felix Krahmer\footremember{TUM}{Department of Mathematics, Technical University of Munich, 85748 
Garching bei M\"unchen, Germany 
(\href{mailto:oleh.melnyk@tum.de}{oleh.melnyk@tum.de},\href{mailto:felix.krahmer@tum.de}{felix.krahmer@tum.de}) } \and Oleh Melnyk\footrecall{HMGU} \footrecall{TUM} 
}
\begin{document}
\selectlanguage{english}

\maketitle
\vspace{-3mm}

\begin{abstract}
The angular synchronization problem of estimating a set of unknown angles from their known noisy pairwise differences arises in various applications. 
It can be reformulated as a optimization problem on graphs involving the graph Laplacian matrix.
We consider a general, weighted version of this problem, where the impact of the noise differs between different pairs of entries and some of the differences are erased completely; this version arises for example in ptychography.
We study two common approaches for solving this problem, 
namely eigenvector relaxation and semidefinite convex relaxation. Although some recovery guarantees are available for both methods, their performance is either
unsatisfying or restricted to the unweighted graphs. We close this gap, deriving recovery guarantees for the weighted problem that are completely analogous to the unweighted version.\vspace{0.5cm}\\

\end{abstract}
\vspace{4mm}

\section{Introduction}\label{sec: Intro}
In this paper we consider the problem of recovering a $d$ dimensional vector of angles $\varphi \in [0,2\pi)^d$ from noisy pairwise differences of its 
entries  $\varphi_\ell - \varphi_j  + \eta_{\ell,j} \MOD 2\pi,\ \ell,j \in \{1,\ldots,d\}$, where $\eta_{\ell,j}$ denotes noise. 
This problem is commonly referred to as \textit{angular synchronization} or \textit{phase synchronization}.
It frequently arises in various applications such as recovery from phaseless measurements \cite{Iwen.2020, Preskitt.2018, Alexeev.2014, Li.2017, Marchesini.2016, Melnyk.7820197122019, Perlmutter.7242019, Forstner.2020}, ordering of data from relative ranks \cite{Cucuringu.2016}, digital communications \cite{Kisialiou.2010}, and 
distributed systems \cite{Giridhar.2006}. The problem of angular synchronization is also closely related to the broader problem of pose graph optimization 
\cite{Carlone.2016}, which appears in robotics and computer vision. \par

Rather than working with the angles $\varphi_\ell$ directly, one typically considers the associated phase factors $\tx x_\ell := e^{i \varphi_\ell},\ \ell \in \{1,\ldots,d\}$. Hence the vector $\tx x=(x_j)_{j=1}^d$ to be recovered belongs to the $d$-dimensional torus 
\[
\Tor^d := \{ v \in \Co^d ~|~ |v_1| = \ldots = |v_d| = 1 \}.
\] 
After this transformation, the pairwise differences $\varphi_\ell - \varphi_j \MOD 2\pi$, $\ell,j \in \{1,\ldots,d\}$ take the form of a product 
\[
e^{i (\varphi_\ell - \varphi_j)} = e^{i \varphi_\ell} \cdot e^{- i \varphi_j} = \tx x_\ell \tx x_j^*, 
\]
where $z^*$ stands for complex conjugate of the number $z$ and complex conjugate transpose in the case of vectors.\\ 
The angular synchronization problem has clearly no unique solution as multiplying the vector $x$ by a factor $e^{i\theta}$ leads to the same product 
$x_\ell x_j^*$. Hence we can at best recover $x$ up to a global phase factor, that is, two solutions $x,x'\in\mathbb{C}^d$ are to be considered equivalent if $y=e^{i\theta}x'$ for some $\theta\in[0,2\pi)$. 
A natural distance measure between two equivalence classes is given by
\begin{equation}
    d(x, x')=\min_{\theta\in[0,2\pi)} \Norm{x-e^{i\theta} x'}{2}.  \label{eq:distance}
\end{equation}
A solution to the angular synchronization problem is thus any vector for which this expression vanishes.

In many applications such as certain algorithms for ptychography \cite{Iwen.2020, Preskitt.2018, Melnyk.7820197122019, Perlmutter.7242019, Forstner.2020}, noisy observations of only a strict subset of the set of differences are available. 
To mathematically describe this restriction we will work with the quantity 
\[
E := \{ (\ell,j) \in \{1,\dots, d\}\times \{1,\dots, d\}:  \text{noisy } \varphi_\ell - \varphi_j \text{ is known and } j\neq \ell\}.
\] 

In these ptychography applications, one also encounters a version of the problem that is generalized in yet another way. Namely, the entries of the vector $y$ to be recovered are not all of modulus $1$ (but still assumed to be known). The measurements are still of the form $y_j y_k^*$ affected by noise. Clearly this generalized problem can be directly reduced to the angular synchronization problem in its original form by just dividing each measurement by the product of the known magnitudes of the associated entries, but one should note that the noise is also affected by this transformation.

We will now present a short overview of the major developments in angular synchronization. The approaches to the problem mainly split into two dominant branches, which essentially differ by the underlying noise model.


In the first branch, it is assumed that the observed pairwise products of the unknown phase factors are affected by independent Gaussian noise. Typically these results work with $E=\{(\ell,j) | j,\ell, \in \{1,\dots, d\}, j\neq \ell\}$, i.e., assuming control of the full set of pairwise differences. 
That is, the matrix of measurements $\est X$ is given by 
\begin{equation}\label{eq: model1}
\est X = \tx x \tx x^* + \sigma\Delta,
\end{equation}
where $\Delta$ is a $d \times d$ Hermitian matrix with $\Delta_{\ell,\ell} = 0$ and $\Delta_{\ell,j},\ \ell > j,$ being independent centered complex Gaussian 
random variables with unit variance and $\sigma>0$. This noise model allows to perform maximum likelihood estimation which leads to the least squares 
problem  (LSP)
\begin{equation}\label{eq: ls formulation}
\min_{z \in \Tor^d} \frac{1}{2} \sum_{ (\ell, j ) \in E} w_{\ell,j} |z_\ell -  \est X_{\ell,j} z_j |^2,
\end{equation}
with weights $w_{\ell,j} = 1/\sigma^2$, $\ell \neq j$ and $w_{\ell,\ell} = 0$. 
Due to the condition $z \in \Tor^d$, the LSP \eqref{eq: ls formulation} is NP-hard \cite{Zhang.2006}. Therefore, Singer \cite{Singer.2011} proposed two possible relaxations, 
the eigenvector relaxation (ER) and the semidefinite convex relaxation (SDP). Both will be discussed in Section \ref{sec: prev results}.  

By a closer inspection of the maximum likelihood approach, Bandeira, Boumal and Singer \cite{Bandeira.2017} were able to establish an error bound for the solution of the LSP \eqref{eq: ls formulation} which holds with high probability. In addition the authors gave sufficient conditions on the standard deviation 
$\sigma$ under which the SDP recovers the solution of the LSP \eqref{eq: ls formulation}. As an alternative to the relaxation approaches Boumal \cite{Boumal.2016} proposed an iterative approach called generalized power method (GPM) to solve the LSP directly. He showed that the method converges to 
the minimizer of \eqref{eq: ls formulation}. Later Liu et al. \cite{Liu.2017} provided additional details about the convergence rate of the GPM. 
In subsequent work \cite{Zhong.2018}, Zhong and Boumal extended the admissible range of $\sigma$ providing near-optimal error bounds for solutions of the LSP, ER and the GPM and improved the sufficient conditions for the tightness of the SDP relaxation. 
 
For the variant of the angular synchronization problem where the vector $y$ to be recovered does not only have entries of modulus one, this theory does not directly apply, as the added Gaussian noise will encounter entrywise rescaling and hence no longer have the same variance for all entries. The least squares approach will, however, have a natural generalization. In analogy to \eqref{eq: ls formulation} where  all differences are multiplied by the inverse of the variance of the i.i.d. noise variables, one weights each difference with the inverse of the variance of the corresponding rescaled noise term, which yields a linear scaling in $|y_j y_\ell|$. While this method is not covered by the theory just discussed, it serves as an important motivation for the approach of this paper.
 
The second branch of development for the angular synchronization problem works with the model that the angular differences rather than the associated phase factors are affected by noise. This version of the problem has also been studied for more general sets $E$.
Consequently, the matrix of measurements $\est X$ in this model is given by the entries 
\begin{equation}\label{eq: X tilde}
\est X_{\ell,j}=
\begin{cases}
e^{i (\varphi_\ell - \varphi_j + \eta_{\ell,j}) }, & (\ell,j) \in E,\\
0, & (\ell,j) \notin E,
\end{cases}
\end{equation}
where $\eta_{\ell,j}$ corresponds to the angular noise, or
\begin{equation}\label{eq: Y tilde}
\est Y_{\ell,j}=
\begin{cases}
|y_\ell y_j| e^{i (\varphi_\ell - \varphi_j + \eta_{\ell,j}) }, & (\ell,j) \in E,\\
0, & (\ell,j) \notin E,
\end{cases}
\end{equation}
when the entries to be recovered are not of unit modulus.

Under this model, random noise is somewhat harder to study due to the multiplicative structure. Consequently, most works employ an adversarial noise model making no assumptions on the distribution of the noise. That is, maximum likelihood estimation is no longer applicable. Nevertheless, 
weighted least squares minimization \eqref{eq: ls formulation} can still be applied without the statistical justification; and a natural choice for the weights remains $w_{j,k} = |y_j y_k|$. This is in line with the observation that  if for two vectors $y$ and $\tilde y$ the modulus of each entry agrees, then smaller entries play less of a role in determining distance in the sense of \eqref{eq:distance}. Indeed the expansion 
\[
d(y, \tilde y)= \min_{\theta\in[0,2\pi)} \Norm{\tilde y-e^{i\theta} y}{2}^2 
=
\min_{\theta\in[0,2\pi)} \sum_{\ell =1}^{d} |y_\ell|^2 |\tilde x_\ell -e^{i\theta} x_\ell |^2,
\]
motivates weights of $|y_\ell|^2$ for the diagonal entries, which naturally generalizes to the weights motivated by the maximum likelihood formula.
For ptychography applications, these weights have also been shown to numerically outperfom the unweighted scenario  (see Section 4.4 in \cite{Preskitt.2018}).
\par

For the multiplicative noise model, several error bounds
have been presented in the literature. Iwen et al. \cite{Iwen.2020} worked with the unweighted LSP \eqref{eq: ls 
formulation} and established recovery guarantees for the ER based on Cheeger's inequality \cite{Bandeira.2013}. Later in \cite{Preskitt.2018}, Preskitt developed 
error bounds for the unweighted case of the LSP. He additionally developed alternative bounds for any selection of weights in the problem \eqref{eq: ls formulation} and 
provided sufficient conditions for tightness of the SDP relaxation. \par

In the literature, the SDP relaxation is studied more often, as under certain conditions it recovers a true solution of the optimization problem 
\eqref{eq: ls formulation}. On the other hand, it is computationally heavy and above a certain noise level the relaxation is longer tight, so SDP fails to return the exact solution of the LSP. 
Thus beyond this threshold, no recovery guarantees for SDP are available.
ER, in contrast, is much faster, especially for large dimension $d$, and its recovery guarantees, where available, are not restricted by tightness assumptions.
Before this paper, however, such guarantees were only available for the unweighted scenarios, even though SDP and ER exhibit similar reconstruction accuracy in numerical experiments.


  
In this paper, we close this gap, providing recovery guarantees for weighted angular synchronization via eigenvector relaxation from measurements of the form \eqref{eq: X tilde}, following the setup of \cite{Iwen.2020, Preskitt.2018}. We numerically demonstrate that our guarantees even outperform the best known guarantees for the unrelaxed problem LSP. Along the way, we also establish improved bounds for LSP.


\section{Problem setup and previous results}\label{sec: prev results}
We study the problem of recovering a vector $x=(x_j)_{j=1}^d$ with unimodular entries $x_j=e^{i\varphi_j}$ from partial and possibly noisy information on the pairwise 
differences $x_\ell x_j^\ast=e^{i(\varphi_\ell-\varphi_j)}$ for all pairs $(\ell,j)$ in some set $E \subset [d]\times [d]$. Here we used the notation $[n]=\{1,\dots, n\}$.
As we consider angular noise, the noisy observations will take the form $e^{i(\varphi_\ell-\varphi_j+\eta_{\ell,j})}$, where $\eta_{\ell,j}\in (-\pi, \pi]$ is the angular noise.

The phase factors corresponding to the true pairwise differences will be arranged as a matrix $X\in\mathbb{C}^{d\times d}$, the noisy observation as a matrix $\hat X\in\mathbb{C}^{d\times d}$, that is, the entries of these matrices are given by
$$
X_{\ell,j}=\begin{cases} e^{i(\varphi_\ell-\varphi_j)} & (\ell,j)\in E,\\ 0,& (\ell,j)\notin E,\end{cases},\qquad 
\hat{X}_{\ell,j}=\begin{cases} e^{i(\varphi_\ell-\varphi_j+\eta_{\ell,j})}& (\ell,j)\in E,\\ 0,& (\ell,j)\notin E.\end{cases}
$$
By  $\hh^d$ we denote the space of all $d \times d$ Hermitian matrices. With $N\in\hh^d$ with  entries $N_{\ell,j}=e^{i\eta_{\ell,j}}$ denoting the matrix rearrangement of the multiplicative noise, one observes that these two matrix representations are related via $\hat{X}=X\circ N$, where for two matrices $A,B\in\Co^{d\times d}$, $A\circ B$ denotes their  Hadamard product as defined by $(A\circ B)_{n,m}=A_{n,m}\, B_{n,m}$.\\ 
As a measure for the noise level, we will use size of $X-\hat X$ in a weighted Frobenius norm or a weighted spectral norm; recall that for $A \in \Co^{d \times d}$  the Frobenius norm and the spectral norm are given by 
\begin{align*}
\Norm{A}{F} := \tr \left( A^*A \right) 
\text{ and }
\Norm{A}{\infty} := \max_{\Norm{v}{2}=1 } \Norm{Av}{2}
\text{, respectively}.
\end{align*}
The quality of reconstruction will be measured in the Euclidean norm on $\Co^d$, given by
\[
\Norm{v}{2} := \left( \sum_{i=1}^{d} |v_i|^2 \right)^{1/2}.
\]
For the proofs we will also need the supremum norm $\Norm{v}{\infty} := \max_{1\leq j \leq d} |v_j|$. 



We will write $A \succeq 0$ if the matrix $A$ is positive semidefinite, that is  
\[
v^* A v \ge 0, \text{ for all } v \in \Co^d.
\] 

The $\sgn$ operator is defined for $\alpha \in \mathbb{C}$ as 
\[
\sgn(\alpha) := 
\begin{cases}
\alpha/|\alpha|, & \alpha\not= 0,\\
0, & \text{otherwise}.
\end{cases}
\]
This operator is extended to any matrix space $\mathbb{C}^{d\times d^\prime}$ by  entrywise operation, i.e. for any $A\in\mathbb{C}^{d\times d^\prime}$
 we have 
$$
\sgn(A)=\big(\sgn(A_{n,m})\big)_{n,m=1}^{d,d^\prime}.
$$

Similarly to previous works, our analysis is based on a graph theoretic interpretation. Namely, the matrices $X$ and $\hat X$ can be seen as edge weight matrix of a weighted undirected graph $G=(V,E,W)$. Consequently, one has $|V|=d$, and we can identify $V$ with $[d]=\{1,\dots, d\}$.
The set of edges $E$ is naturally identified with the index set of the observed noisy angular differences introduced above. It directly follows from the problem setup that the weight function $W:V\times V\to [0,\infty)$ must satisfy $W(v,v^\prime)\geq 0,\ W(v,v^\prime)=W(v^\prime,v),\ W(v,v)=0$. 


To analyze this graph, we need some basic concepts from graph theory.
The adjacency matrix $A_G$ of $G$ is given by 
$$
(A_G)_{\ell,j}=\begin{cases} 1 & (\ell,j)\in E,\\ 0,& (\ell,j)\notin E,\end{cases}.
$$
With  this notation, one obtains the compact expression $X= A_G(xx^*)$.
In case $W\equiv 1$ on its support, i.e., $W=A_G$, we speak of $G$ as an unweighted graph. 

The degree of the vertex $\ell$ is defined as 
$$
\deg (\ell) := \sum_{(j,\ell)\in E} w_{\ell,j},
$$
and the corresponding degree matrix is the diagonal matrix 
$$
D=\diag\big(\deg(\ell)\big).
$$
The Laplacian of the graph $G$ is given by 
$$
L_G=D-W. 
$$
Observe that, as the graph is undirected, the Laplacian is symmetric. Moreover, since $w_{\ell,j}\geq 0$ we have 
$$
u^\ast L_G u=\frac{1}{2}\sum_{\ell,j} w_{\ell,j}\, |u_\ell - u_j|^2\geq 0
$$
for all $u\in\Co^d$. Hence the Laplacian is positive semidefinite and therefore has a spectrum consisting of non-negative real numbers, which we 
denote by $\lambda_j$ with indices $j$ arranged in ascending order, i.e., 
$$
0=\lambda_1\leq \lambda_2\leq\cdots\leq\lambda_d.
$$
Here the first equality follows from the observation  that the vector $\mathds{1}=(1,\dots,1)^T$ satisfies $L_G\mathds{1}=0$.
The spectral gap of $G$ is defined as $\tau_G=\lambda_2$. A graph $G$ is connected if and only if $\tau_G>0 $, see  \cite{Chung.1997}. In that case, the null space of $L_G$ is spanned by $\mathds{1}$. \\
Besides the Laplacian $L_G$ the normalized Laplacian $L_N$  of $G$ is often used. It is defined as 
$$
L_N=D^{-1/2}\, L_G\,  D^{-1/2}. 
$$
Its spectrum consists of non-negative real numbers as well and we  write $\tau_N$ for its second smallest eigenvalue $\lambda_2(L_N)$. 
 

The data dependent Laplacians associated to $X$ and $\hat{X}$ are defined as 
$$
L=D-W\circ X,\quad \text{ and }\quad \hat{L}=D-W\circ\hat{X}, \text{ respectively.} 
$$
Note that under the multiplicative noise model used in this paper, both these Laplacians are positive semidefinite matrices by Gershgorin's circle theorem.

The data dependent Laplacian $\hat L$ corresponding to the noisy observations allows for a compact representation of the least squares problem \eqref{eq: ls formulation} at the core of our recovery method. Indeed, observe that

\begin{equation}\label{eq: L formulation}
\min_{z \in \Tor^d} \frac{1}{2}\sum_{ (\ell, j ) \in E} w_{\ell,j} |z_\ell -  \est X_{\ell,j} z_j |^2
=
\min_{z \in \Tor^d} z^* (D - \est X \circ W) z
=
\min_{z \in \Tor^d} z^* \est L z.
\end{equation}
Due to the quadratic constraint $z\in\mathbb{T}^d$ the quadratic minimization problem \eqref{eq: L formulation} is non-convex and thus NP-hard in general. 
One way to obtain a feasible problem is to relax the constraint in \eqref{eq: L formulation} to $\Norm{z}{2}^2 = d$ and obtain 
\begin{equation}
\min_{ \Norm{z}{2}^2 = d} z^* \est L z.\label{eq: L eig}
\end{equation}
This is nothing else but the determination of the smallest eigenvalue of the matrix $\hat{L}$ and can be solved efficiently using Rayleigh quotient. We will 
refer to \eqref{eq: L eig} as eigenvector relaxation (ER). \\
An error bound for the ER based reconstruction was given by Iwen et. al. in \cite{Iwen.2020} for the case of unweighted graphs. 
Their proof is based on the Cheeger inequality that is only available for the normalized Laplacian, which is why the minimization problem in their theorem has a different normalization than \eqref{eq: L eig}. In the special case that $\deg(\ell)$ is a constant for all $\ell$ (as in \cite{Iwen.2020}), the two normalizations agree up to a constant. Using the terminology introduced above their result reads as follows.

\begin{theorem}[{\citep[Theorem 3]{Iwen.2020}, \citep[Theorem 4]{Preskitt.2018}}] \label{thm: weak bound}
Suppose that $G = (V,E)$ is an undirected connected and unweighted graph with $\tau_N > 0$. 
Let $\tilde{z} \in \Co^d$ be the minimizer of 
\[
\min_{\Norm{z}{2}^2 = d} z^* D^{-1/2} \est L D^{-1/2} z 
\] 
and let $\tilde{x} = \sgn(\tilde{z})$.
Then,
\begin{equation}
\min_{\theta \in [0, 2 \pi)} \Norm{\tilde{x} - e^{i\theta} \tx x}{2} \le 19 \frac{\Norm{\est X - \tx X}{F}}{\tau_N \sqrt{\min\limits_{i \in V} \deg(i) } }. \label{eq:prevworkunw}
\end{equation}
\end{theorem}


An alternative approach is based on the idea of lifting the problem to the matrix space. It makes use of the relation 
$$
z^*\hat{L}z=\sum_{\ell,j} z_\ell^*\, L_{\ell,j}\, z_j=\tr(\hat{L}\,  z z^\ast). 
$$
With this the minimization problem \eqref{eq: L formulation} transforms into 
\begin{align}\label{eq: L Lift}
\min_{ Z \in \hh^d } \tr( \est L Z) \nonumber \\
s.t.~ Z_{ii} = 1,\\
Z \succeq 0, \nonumber \\
\rank(Z) = 1. \nonumber
\end{align}
The class of minimization problems with explicit rank constraints is known to include many NP-hard instances \cite[Chapter 2]{Fazel.March2002}, so a common strategy is to perform a semi-definite relaxation. For \eqref{eq: L Lift}, the following relaxation has been proposed in \cite{Goemans.1995}.

\begin{align}\label{eq: L SDP}
\min_{ Z \in \hh^d } \tr( \est L Z) \nonumber \\
s.t.~ Z_{ii} = 1,\\
Z \succeq 0. \nonumber 
\end{align}
We will refer to this minimization problem as SDP. 
Note that if $Z$ meets the rank condition in \eqref{eq: L Lift} one obtains that $Z= z z^*$, where $z$ is a solution of \eqref{eq: L formulation}. Without the rank condition, however, the solution to \eqref{eq: L SDP} may have higher rank. In this case the methods outputs the phase factors corresponding to the entries of the eigenvector associated to the largest 
eigenvalue as an approximation for the solution of \eqref{eq: L formulation} \cite{So.2007}. \\

As it was mentioned before, the error bound is commonly derived for the solution of the LSP and then applied to the solution of the SDP when it has rank 
$1$. For unweighted graphs, a first result on recovery guarantees of the SDP has been established by Preskitt \cite{Preskitt.2018}. Adjusted to our 
terminology his result reads as follows.

\begin{theorem}[{\citep[Theorem 9]{Preskitt.2018}}]\label{thm: SDP}
Suppose that $G = (V,E)$ is an undirected and unweighted graph with $\tau_G>0$. 
Let $\tilde{x} \in  \Tor^d$ be the minimizer of the LSP \eqref{eq: L formulation}. 
Then,
$$
\min_{\theta \in [0, 2 \pi)} \Norm{\tilde{x} - e^{i\theta} \tx x}{2} \le 2 \frac{\Norm{\est X - \tx X}{F}}{\sqrt { \tau_G }}.
$$
If additionally inequality
$$
\Norm{\est X -  \tx X }{F} < \frac{\tau_G}{1 + \sqrt{d}}
$$
holds, then $\tilde{x} \tilde{x}^*$ is the minimizer of the SDP \eqref{eq: L SDP}.
\end{theorem}

For a detailed comparison of Theorems \ref{thm: weak bound} and \ref{thm: SDP}, we refer reader to Section 4.3.2 of \cite{Preskitt.2018}. 


The first results addressing a generalization to the important case of weighted graphs have been derived by Preskitt \cite{Preskitt.2018}. The following formulations have again been adjusted to our notation. 



\begin{theorem}[{\citep[Proposition 12 and Theorem 8]{Preskitt.2018}}]\label{thm: SDP weighted}
Suppose that $G = (V,E,W)$ weighted graph with $\tau_G>0$. 
Let $\tilde{x} \in  \Tor^d$ be the minimizer of the LSP \eqref{eq: L formulation}.
Then,
\begin{equation}\label{eq: SqrtEst} \tag{2.3.A}
\min_{\theta \in [0, 2 \pi)} \Norm{\tilde{x }- e^{i\theta} \tx x}{2} \le 2 \sqrt{ \frac{d \Norm{ W \circ (\est X - \tx X) }{\infty}}{ \tau_G } },
\end{equation}
and 
\begin{equation}\label{eq:prevworkweights} \tag{2.3.B}
\min_{\theta \in [0, 2 \pi)} \Norm{\tilde{x} - e^{i\theta} \tx x}{2} \le 4 \sqrt{d}\,  \frac{\Norm{ W \circ (\est X - \tx X)}{\infty}}{ \tau_G }.
\end{equation}
If additionally inequality 
$$
\Norm{W \circ (\est X - \tilde{x} \tilde{x}^*) }{\infty} < \frac{\tau_G}{1 + \sqrt{d}}
$$
holds, then $\tilde{x} \tilde{x}^*$ is the minimizer of the SDP \eqref{eq: L SDP}.
\end{theorem}


As the square root in \eqref{eq: SqrtEst} produces slow convergence as the noise diminishes, i.e. $\hat{X}$ approaches $X$, in the many cases bound \eqref{eq:prevworkweights} outperforms \eqref{eq: SqrtEst}. For unweighted graphs, as we numerically explore in Section \ref{sec: numerics}, the bounds of Theorem \ref{eq:prevworkweights} 
are similar to those of
Theorem \ref{thm: SDP} and superior to the bounds of Theorem \ref{thm: weak bound} for ER in many cases. They are, however, only valid for SDP when the relaxation is tight,
which is only guaranteed when the weighted noise on the phase factor matrix is bounded by $\tau_G/(1 + \sqrt d)$. As the spectral gap $\tau_G$ is typically rather small, as compared to the dimension $d$, tightness is guaranteed only for very small noise levels. In fact, our numerical simulations in Section \ref{sec: numerics} show that the SDP relaxation is indeed not tight in many cases. In constrast, the recovery guarantees for ER provided by Theorem \ref{thm: weak bound} for the unweighted case are applicable independently of the tightness of the relaxation.

\section{Improved error bounds}\label{sec: results}

The main contribution of this paper concerns recovery guarantees for weighted angular synchronization via eigenvector relaxation, which are often stronger than even the best known bounds for the unrelaxed problem and do not require any a priori bound for the error to ensure tightness of the relaxation. Along the way, we derive similar error bounds for the solution of the least squares problem, which are exactly analogous to those provided by Theorem \ref{thm: SDP} in the unweighted case.
The superior scaling of our error bounds as compared to Theorem \ref{thm: SDP weighted} is also confirmed by numerical simulations in Section~\ref{sec: numerics}. We first state our result in general form, before discussing three special cases of interest.

\begin{theorem}\label{thm: weighted new}
Suppose that $G = (V ,E,W)$ is a weighted graph with $\tau_G>0$. 
Let $\tilde{x} \in  \Tor^d$ be the minimizer of the LSP \eqref{eq: L formulation} and $z$ be the minimizer of the ER $\eqref{eq: L eig}$.
Set $R\in\mathbb{C}^{d\times d}$ as $R_{\ell,j}=W_{\ell,j}^{1/2}$ . 
Then,
\begin{equation}\label{eq: Est1}
\min_{\theta \in [0, 2 \pi)} \Norm{\tilde{x}- e^{i\theta} \tx x}{2} \le 2\ \frac{\Norm{R \circ (\est X - \tx X)}{F}}{ \sqrt{ \tau_G } },
\end{equation}
and
\begin{equation}\label{eq: Est2}
\min_{\theta \in [0, 2 \pi)} \Norm{\sgn(z) - e^{i\theta} \tx x}{2} \le 2\ \frac{ c_z \ \Norm{R \circ (\est X - \tx X)}{F}}{ \sqrt{ \tau_G } },
\end{equation}
with $c_z := \sqrt{2 + 2 \Norm{z}{\infty}^2 }$. 
If additionally inequality
$$
\Norm{ W \circ (\est X - \tilde{x} \tilde{x}^*) }{\infty} < \frac{\tau_G}{1 + \sqrt{d}}
$$
holds, then $\tilde{x}\tilde{x}^*$ is the minimizer of \eqref{eq: L SDP}.
\end{theorem}

We note that $\Norm{R \circ (\est X - \tx X)}{F}$ can be estimated from above as 
\begin{equation}\label{eq: Rem}
\Norm{R \circ (\est X - \tx X)}{F} \le \sqrt{ \Norm{W \circ (\est X - \tx X)}{F} \ \Norm{ \est X - \tx X}{F} }.
\end{equation}

The first consequence of interest concerns the unweighted case, where the noise norm in Theorem \ref{thm: weighted new} 
simplifies to $\Norm{\est X - \tx X}{F}$ as in Theorem \ref{thm: SDP}. 

For LSP, Theorems \ref{thm: weighted new} and \ref{thm: SDP} yield the exact same bound, the bounds in Theorem \ref{thm: weighted new} for ER only differ by multiplicative factor $c_z$, which is $2$ when relaxation is tight and $\sqrt{2+2d}$ in the worst case. 


\begin{corollary}\label{cor: eig unweighted}
Suppose $G = (V,E)$ is an unweighted graph with $\tau_G>0$. Let $z \in \Co^d$ be the minimizer of the ER \eqref{eq: L eig}.
Then
$$
\min_{\theta \in [0, 2 \pi)} \Norm{\sgn(z) - e^{i\theta} \tx x}{2} \le 2 \frac{ c_z \ \Norm{\tx X - \est X}{F}}{\sqrt{\tau_G} }.
$$
\end{corollary}

Our next examples are related to the ptychography problem. 
We remind the reader that the goal of ptychography is to estimate the a signal $y \in \Co^d$ from phaseless measurements through localized masks. A recent method for recovering the signal $y$ from such observation is the BlockPR algorithm by Iwen et al.~\cite{Iwen.2020}, see \cite{Preskitt.2018, Melnyk.7820197122019,Perlmutter.7242019,  Forstner.2020} for follow-up works developing this algorithm further that also rely on weighted angular synchronization. The BlockPR algorithm proceeds by combining neighboring masks to obtain estimates for the products of entries located close to each other. In mathematical terms, this procedure yields an approximation of the squared absolute values of the entries (so these can be assumed to be approximately known) and a noisy version of $T_\delta( y y^*)$, where $\delta$ is the size of the mask and $T_\delta$ is the restriction operator mapping a matrix to its entries indexed by the set 
\begin{equation}\label{eq: E delta}
E_\delta =\{ (\ell,j) ~|~ \ell \neq j\in [d], \text{ and }  \left| \ell - j \right| < \delta \text{ or } |\ell-j| > d - \delta \}
\end{equation}
corresponding to the $2\delta-1$ central sub- and superdiagonals, excluding the entries on the main diagonal. 
Thus the resulting measurements exactly correspond to \eqref{eq: Y tilde} for $E=E_\delta$, which is why weighted angular synchronization is the natural method of choice. The weights in this problem are given by the matrix $yy^*$ restricted to the index set $E$, which yields the setup of the following corollary.


\begin{corollary}\label{cor: weights normal new}
Consider a weighted graph $G=(V,E,W)$ with whose weight matrix $W$ is defined as follows.
Let $y\in\mathbb{C}^d$ with $\sgn(y)=x$. Define matrices $Y=(I+A_G)\circ y\, y^\ast $ and $\tx X = \sgn(\tx Y)$. Let $M$ and $\est X$ be the matrices containing the perturbed magnitudes and phases of $Y$, respectively, so that $M \approx |\tx Y|$ and $\est X = \tx X \circ N$ and set $\hat{Y}= M \circ \est X$. Consider the weight matrix $W$ with entries given by $w_{\ell,j}=|\hat{Y}_{\ell,j}|\, (1-\delta_{\ell,j})$ and assume that $\tau_G>0$. Let $\tilde{x}$ be a minimizer of \eqref{eq: L formulation} and 
let $z$ be the minimizer of \eqref{eq: L eig}. Then we have   
$$
\min_{\theta \in [0, 2 \pi)} \Norm{\tilde{x} - e^{i\theta} \tx x}{2} \le 2 \sqrt2 \frac{ \sqrt{ \Norm{\est Y - \tx Y}{F} \ \Norm{\est X - \tx X}{F} }}{ \sqrt{ \tau_G } },
$$
and
$$
\min_{\theta \in [0, 2 \pi)} \Norm{\sgn(z) - e^{i\theta} \tx x}{2} \le  2 \sqrt 2 \frac{c_z \ \sqrt{  \Norm{\est Y - \tx Y}{F} \ \Norm{\est X - \tx X}{F} } }{ \sqrt{ \tau_G } }.
$$

\end{corollary}
In the next statement the set-up is analogous to the previous corollary but instead of having weights defined by $|\hat{Y}_{\ell,j}|$ we work with  
$|\est Y_{\ell,j}|^2$. 
\begin{corollary}\label{cor: squared weights}
Consider a weighted graph $G=(V,E,W)$ whose weight matrix $W$ is defined as follows.
Let $y\in\mathbb{C}^d$ with $\sgn(y)=x$. Define matrices $Y=(I+A_G)\circ y\, y^\ast $ and $\tx X = \sgn(\tx Y)$. Let $M$ and $\est X$ be the matrices containing the perturbed magnitudes and phases of $Y$, respectively, so that $M \approx |\tx Y|$ and $\est X = \tx X \circ N$ and set $\hat{Y}= M \circ \est X$. Consider the weight matrix $W$ with entries given by $w_{\ell,j}=|\hat{Y}_{\ell,j}|^2\, (1-\delta_{\ell,j})$ and assume that $\tau_G>0$. Let $\tilde{x}$ be a minimizer of \eqref{eq: L formulation} and let 
$z$ be the minimizer of  \eqref{eq: L eig}. 
Then we have  
$$
\min_{\theta \in [0, 2 \pi)} \Norm{\tilde{x} - e^{i\theta} \tx x}{2} \le 4 \frac{ \Norm{\est Y - \tx Y}{F}}{ \sqrt{ \tau_G } },
$$
and
$$
\min_{\theta \in [0, 2 \pi)} \Norm{\sgn(z) - e^{i\theta} \tx x}{2} \le  4 \frac{c_z \ \Norm{\est Y - \tx Y}{F}  }{ \sqrt{ \tau_G } }.
$$
\end{corollary}

\section{ Proofs}

\begin{proof}[Proof of Theorem \ref{thm: weighted new}]
We will proceed by establishing the following four inequalities.
\begin{equation} \label{eq: in1}
\min_{\theta \in [0, 2 \pi)} \Norm{\tilde{x} - e^{i\theta} \tx x}{2}^2 \le \frac{1}{\tau_G}\sum_{(\ell,j) \in E} w_{\ell,j} |\tx x_\ell^* \tilde{x}_\ell - \tx x_\ell^* \tilde{x}_j|^2,
\end{equation}
\begin{equation} \label{eq: in1 eig}
\min_{\theta \in [0, 2 \pi)} \Norm{\sgn(z) - e^{i\theta} \tx x}{2}^2 \le \frac{4}{\tau_G}\sum_{(\ell,j) \in E} w_{\ell,j} | \tx x_\ell^* z_\ell -   \tx x_j^* z_j |^2 ,
\end{equation}
\begin{equation} \label{eq: in2 eig}
\sum_{(\ell,j) \in E} w_{\ell,j} | \tx x_\ell^* z_\ell -   \tx x_j^* z_j |^2  \le c_z^2\Norm{R \circ (\est X - \tx X)}{F}^2.
\end{equation}
and
\begin{equation} \label{eq: in2}
\sum_{(\ell,j) \in E} w_{\ell,j} |\tx x_\ell^* \tilde{x}_\ell - \tx x_\ell^* \tilde{x}_j|^2 \le 4 \Norm{R \circ (\est X - \tx X)}{F}^2,
\end{equation}
Note that Inequality \eqref{eq: in1} has been derived in \cite{Preskitt.2018}, we will nevertheless include a proof for completeness.

Equation \eqref{eq: Est1} then follows by combining \eqref{eq: in1} and \eqref{eq: in2}, Equation \eqref{eq: Est2} is obtained as a combination of \eqref{eq: in1 eig} and \eqref{eq: in2 eig}. 
The condition for $xx^*$ to be a minimizer of \eqref{eq: L SDP} directly follows from the following Lemma.

\begin{lemma}[{\citep[Lemma 16]{Preskitt.2018}}]\label{l: 16}
Suppose $\tilde{x}\in\mathbb{T}^d$ is a minimizer of \eqref{eq: L formulation} and let $\tilde{L}=D-W\circ \tilde{x}\, \tilde{x}^*$. If 
$$
\Norm{\est L - \tilde L}{F} < \frac{\tau_G}{1 + \sqrt{d}},
$$
then $\tilde{x} \tilde{x}^*$ is minimizer of \eqref{eq: L SDP}.
\end{lemma}

It remains to prove the four inequalities.
To that extent, we recall that for $\alpha,\beta \in \Co$ with $|\beta| = 1$ we have
\begin{align}
| \sgn(\alpha) - \beta| 
& \le |\alpha - \beta | + | \sgn(\alpha) - \alpha | = | \alpha - \beta | + | 1 - |\alpha | | \nonumber \\
& = | \alpha - \beta | + | |\beta | - |\alpha| | \le 2| \alpha - \beta |. \label{eq: sgn in}
\end{align}
With help of this inequality we obtain that
\begin{align*}
\min_{\theta \in [0, 2 \pi]} \Norm{\sgn(z) - e^{i\theta} \tx x}{2}^2 
&   = \min_{\theta \in [0, 2 \pi]} \sum_{\ell = 1}^{d} | \sgn(z_\ell) - e^{i\theta} \tx x_\ell|^2 \nonumber \\
& \le 4 \min_{\theta \in [0, 2 \pi]} \sum_{\ell = 1}^{d} | z_\ell - e^{i\theta} \tx x_\ell|^2 \\
& = 4 \min_{\theta \in [0, 2 \pi]} \Norm{z - e^{i\theta} \tx x}{2}^2.
\end{align*}
Moreover, since $\Norm{x}{2}^2=d$ and $\Norm{z}{2}^2=d$ we have that
\begin{equation} \label{eq: norm decomposition}
\Norm{z - e^{i\theta} \tx x}{2}^2  = \Norm{z}{2}^2 + \Norm{ e^{i\theta} \tx x }{2}^2 - 2 \RE \left( e^{- i \theta} \tx x^* z \right) = 2d - 2 \RE \left( e^{- i \theta} \tx x^* z \right).
\end{equation}
The right hand side is minimal if $\RE \left( \tx x^* z \right)$ is maximal and equal to $| \tx x^* z |$.
Hence with $e^{i \vartheta} := \sgn(\tx x^* z)^*$ we arrive at  
\begin{equation} \label{eq: real part}
\RE \big( e^{- i \vartheta} \tx x^* z \big) = \RE \big( \sgn(\tx x^* z)^* \cdot \sgn(\tx x^* z) \cdot | \tx x^* z |  \big) = | \tx x^* z |,
\end{equation}
which motivates employing the following inequality 
\begin{equation}\label{eq: sgn}
\min_{\theta \in [0, 2 \pi)} \Norm{\sgn(z) - e^{i\theta} \tx x}{2}^2  \le 4 \Norm{z - e^{i \vartheta} \tx x}{2}^2.
\end{equation}
Define the unitary matrix  $C=\diag\{x_1,\dots, x_d\}$, where $x=(x_j)_{j=1}^d$. Note that   
\[
W \circ X = W \circ (x x^*) =  C W C^*
\]
and by union modularity of $x_i$
\[
D = D C C^* = C D C^*, 
\]
where we used commutativity of diagonal matrices. This results in
$$
\tx L = D - W \circ \tx X = C D C ^* - C W C^* =C L_G C^*,
$$
which shows in particular that the eigenvalues of $\tx L$ and $L_G$ coincide. \\
By assumption $\tau_G>0$ and hence the null space of $L_G$ is spanned by $\mathds{1}$. Thus  the null space of $L$ is spanned 
by $x$.  \\
The projection of $e^{i\vartheta}z$ onto the orthogonal complement of $x$ is given by 
$$
q := 
e^{- i \vartheta} z -   \big\langle e^{- i \vartheta} z, \frac{\tx x}{\Norm{\tx x}{2}}  \big\rangle\, \frac{\tx x}{\Norm{\tx x}{2}}
= e^{- i \vartheta} z - \frac{1}{d} \tx |\tx x^* z|\, x,
$$ 
where we used that by \eqref{eq: real part}, the inner product is real.
Consequently, as $q\perp x$, one has that by Pythagoras' theorem
\begin{equation} \label{eq: norm y}
\Norm{q}{2}^2 
= \Norm{e^{- i \vartheta} z}{2}^2 - \Norm{\frac{1}{d} \tx x |\tx x^* z|}{2}^2 
= \Norm{z}{2}^2- \frac{1}{d^2} \Norm{\tx x}{2}^2 |\tx x^* z|^2 = d - \frac{1}{d} |\tx x^* z|^2,
\end{equation}
and as $x$ is in the null space of $L$
$$
q^* \tx L q = (e^{- i \vartheta} z - \frac{1}{d} \tx x \cdot |\tx x^* z|)^* \tx L (e^{- i \vartheta} z - \frac{1}{d} \tx x \cdot |\tx x^* z|) = z^* \tx L z.
$$
In view of \eqref{eq: norm decomposition} and \eqref{eq: real part}, we have
$$
 \Norm{e^{- i \vartheta} z - \tx x}{2}^2 = 2(d - | \tx x^* z |).
$$
With the Cauchy-Schwarz inequality and \eqref{eq: norm y}, this yields that
$$
\Norm{q}{2}^2  
= d - \frac{1}{d} |\tx x^* z|^2 \ge  
d - |\tx x^* z| 
= \frac{1}{2} \Norm{z - e^{i \vartheta} \tx x}{2}^2.
$$
By definition $q$ is orthogonal to the null space of $\tx L$ which implies 
$$
z^* \tx L z = q^* \tx L q \ge \lambda_2(\tx L) \Norm{q}{2}^2  \ge  \frac{\lambda_2(\tx L)}{2} \Norm{z - e^{i \vartheta} \tx x}{2}^2.
$$
Combining this with \eqref{eq: sgn} and the fact that $\lambda_2(\tx L) = \tau_G$ as well as the definition  of $L$, we obtain both \eqref{eq: in1} and \eqref{eq: in1 eig} by 
\begin{align*}
& \min_{\theta \in [0, 2 \pi)} \Norm{\sgn(z) - e^{i\theta} \tx x}{2}^2  
\le 4 \Norm{z - e^{i \vartheta} \tx x}{2}^2 
\le \frac{8}{\tau_G} z^* \tx L z \\
& \quad
= \frac{8}{\tau_G}  \frac{1}{2} \sum_{ (\ell, j ) \in E} w_{\ell,j}  |z_\ell - \tx X_{\ell,j} z_j |^2 
= \frac{4}{\tau_G}  \sum_{ (\ell, j ) \in E}  w_{\ell,j}  |z_\ell -  \tx x_\ell \tx x_j^* z_j |^2 \\
& \quad 
= \frac{4}{\tau_G}  \sum_{ (\ell, j ) \in E}  w_{\ell,j}  | \tx x_\ell^* z_\ell -   \tx x_j^* z_j |^2.
\end{align*}
Indeed, \eqref{eq: in1} follows by comparing the second and the last item in this chain of inequalities, and \eqref{eq: in1 eig} by comparing the first and the last item.

Now we will prove inequality \eqref{eq: in2 eig} and  \eqref{eq: in2}, again with largely identical proofs.
For simplicity of notation, we introduce the following auxiliary variables 
$$
g_\ell := \tx x_\ell^* \tilde{x}_\ell, \  h_\ell := \tx x_\ell^* z_\ell,  \text{ and } \Lambda_{\ell,j} := \tx X_{\ell,j}^* \est X_{\ell,j} .
$$

We start by using $(\alpha + \beta)^2 \le 2\alpha^2 + 2 \beta^2 $ to get
\[
| h_\ell -  h_j |^2 
=| h_\ell - \Lambda_{\ell,j} h_j + \Lambda_{\ell,j} h_j -  h_j |^2 
\le 2| h_\ell - \Lambda_{\ell,j} h_j|^2 + 2|h_j|^2 |\Lambda_{\ell,j} - 1 |^2,
\]
and we further estimate
$$
\sum_{ (\ell, j ) \in E} w_{\ell,j} | h_\ell -  h_j |^2 
\le 2 \sum_{ (\ell, j ) \in E} w_{\ell,j} | h_\ell - \Lambda_{\ell,j} h_j|^2 
+ 2 \sum_{ (\ell, j ) \in E} w_{\ell,j} |h_j|^2 |\Lambda_{\ell,j} - 1 |^2.
$$

For the first sum we observe that
\[
| h_\ell - \Lambda_{\ell,j} h_j| 
= | \tx x_\ell^* z_\ell - \tx X_{\ell,j}^* \est X_{\ell,j} \tx x_j^* z_j|
= | \tx x_\ell^* z_\ell - \tx x_\ell^* \tx x_j \est X_{\ell,j} \tx x_j^* z_j|
= | z_\ell - \est X_{\ell,j} z_j|, 
\]
and obtain using \eqref{eq: L formulation} and the fact that $z$ minimizes \eqref{eq: L eig} that 
\begin{align}
\sum_{ (\ell, j ) \in E} w_{\ell,j}  | h_\ell - \Lambda_{\ell,j} h_j|^2
& = \sum_{ (\ell, j ) \in E} w_{\ell,j}  | z_\ell - \est X_{\ell,j} z_j|^2 
= 2 z^* \est L z 
\le 2 \tx x^* \est L \tx x \nonumber\\
& = \sum_{ (\ell, j ) \in E} w_{\ell,j}  | \tx x_\ell - \est X_{\ell,j} \tx x_j|^2 
= \sum_{ (\ell, j ) \in E} w_{\ell,j}  | \tx x_\ell \tx x_j^* - \est X_{\ell,j} |^2 \nonumber \\
& = \sum_{ (\ell, j ) \in E} w_{\ell,j} | \tx X_{\ell,j} - \est X_{\ell,j} |^2
= \Norm{ R \circ (\est X - \tx X)}{F}^2. \label{eq:Rbound}
\end{align} For the second sum we use that $|h_j| = |\tx x_j^* z_j| = |z_j|$ and obtain
\begin{align*}
\sum_{ (\ell, j ) \in E} w_{\ell,j} |h_j|^2  |\Lambda_{\ell,j} - 1 |^2 
& \le \max_{j \in [d] } |h_j|^2  \sum_{ (\ell, j ) \in E} w_{\ell,j}  |\Lambda_{\ell,j} - 1 |^2  = \Norm{z}{\infty}^2  \sum_{ (\ell, j ) \in E} w_{\ell,j}  |\Lambda_{\ell,j} - 1 |^2.
\end{align*} 
The last step is to notice that
\begin{align*}
\sum_{ (\ell, j ) \in E} w_{\ell,j}  |\Lambda_{\ell,j} - 1 |^2 
= \sum_{ (\ell, j ) \in E} w_{\ell,j}   | \tx X^*_{\ell,j} \est X_{\ell,j} - 1 |^2 = \\
\quad = \sum_{ (\ell, j ) \in E} w_{\ell,j}   | \tx X_{\ell,j} -  \est X_{\ell,j} |^2  
= \Norm{R \circ (\est X - \tx X)}{F}^2.
\end{align*} 
Putting everything together we arrive at 
\begin{align*}
\sum_{ (\ell, j ) \in E} w_{\ell,j} | h_\ell -  h_j |^2 
& \le 2 \sum_{ (\ell, j ) \in E} w_{\ell,j} | h_\ell - \Lambda_{\ell,j} h_j|^2 
+ 2 \sum_{ (\ell, j ) \in E} w_{\ell,j} |h_j|^2 |\Lambda_{\ell,j} - 1 |^2 \\
& \le  2 \Norm{R \circ (\est X - \tx X)}{F}^2 
+ 2  \Norm{z}{\infty}^2 \Norm{ R \circ (\est X - \tx X)}{F}^2 
 = c_z^2 \ \Norm{ R \circ (\est X - \tx X)}{F}^2.
\end{align*} 
This concludes the proof of Inequality \eqref{eq: in2 eig}.\\ 
For Inequality \eqref{eq: in2} we proceed analogously, with $\tilde x$ taking the role of $z$; the only difference is that in \eqref{eq:Rbound} we are using the fact that $\tilde x$ minimizes \eqref{eq: L formulation} rather than the fact that $z$ minimizes \eqref{eq: L eig}. The bound for the second sum is simplified as compared to \eqref{eq: in2 eig}, as we replaced $\Norm{z}{\infty}$ by $\Norm{\tilde x}{\infty} = 1$.

The combined bound reads as 
\begin{align*}
\sum_{ (\ell, j ) \in E} w_{\ell,j} \, | g_\ell -  g_j |^2  
& \le 2 \sum_{ (\ell, j ) \in E} w_{\ell,j} \, |\Lambda_{\ell,j} - 1 |^2  
+ 2 \sum_{ (\ell, j ) \in E} w_{\ell,j} \, | g_\ell - \Lambda_{\ell,j} g_j|^2 \\
& \le 4 \Norm{R \circ (\est X - \tx X)}{F}^2.
\end{align*} 

\end{proof}

\begin{proof}[Proof of Corollary \ref{cor: eig unweighted}.]
For an unweighted graph $G$ we immediately get 
\begin{align*}
\Norm{R \circ (\est X - \tx X)}{F}^2 
= \sum_{ (\ell, j ) \in E} 1 \cdot | \tx X_{\ell,j} -  \est X_{\ell,j} |^2 
= \Norm{\est X - \tx X}{F}^2.
\end{align*}
\end{proof}

\begin{proof}[Proof of Corollary \ref{cor: weights normal new}.]
Define an auxiliary weight matrix $W_0$ by $(W_0)_{\ell,j}=|Y_{\ell,j}|\, (1-\delta_{\ell,j})$. 
Using inequality \eqref{eq: Rem} we obtain 
\begin{align*}
\Norm{W \circ (\est X - \tx X)}{F}
& = \Norm{W \circ \est X - W \circ \tx X}{F}\\
& \le \Norm{W \circ \est X - W_0 \circ \tx X}{F} + \Norm{W_0 \circ \tx X - W \circ \tx X}{F}\\
& \le \Norm{\hat{Y} -  Y}{F} + \Norm{ (W_0 - W) \circ \tx X}{F}\\
& = \Norm{\hat{Y} -  Y}{F} + \Norm{ W_0 - W }{F}\\ 
& \le \Norm{\hat{Y} -  Y}{F} + \Norm{\hat{Y} - Y}{F}  = 2\Norm{\hat{Y} - \tx Y}{F},
\end{align*}
where in the third line we only increased the number of non-negative summands by adding the diagonal elements and in the last line we used the  
inequality $||\alpha |-|\beta || \le |\alpha - \beta |$. 
\end{proof}

\begin{proof}[Proof of Corollary \ref{cor: squared weights}.]
We rewrite the right side of the bound in Theorem \ref{thm: weighted new} as
\begin{align*}
\Norm{R \circ (\est X - \tx X)}{F}^2 
& = \sum_{ (\ell, j ) \in E} w_{\ell,j} \,  | \tx X_{\ell,j} -  \est X_{\ell,j} |^2 
= \sum_{ (\ell, j ) \in E} |\est Y_{\ell,j}|^2 \, | \tx X_{\ell,j} -  \est X_{\ell,j} |^2 \\
& = \sum_{ (\ell, j ) \in E} |\est Y_{\ell,j}|^2 \, | \sgn( \tx Y_{\ell,j}) -  \est X_{\ell,j} |^2 \\
& = \sum_{ (\ell, j ) \in E} |\est Y_{\ell,j}|^2 \, \left| \sgn\left( \frac{ \tx Y_{\ell,j}}{|\est Y_{\ell,j}|} \right) -  \est X_{\ell,j} \right|^2, \\
\end{align*}
and apply the inequality \eqref{eq: sgn in} to get 
\begin{align*}
\Norm{R \circ (\est X - \tx X)}{F}^2
& \le \sum_{ (\ell, j ) \in E} 4 |\est Y_{\ell,j}|^2 \, \left| \frac{ \tx Y_{\ell,j}}{|\est Y_{\ell,j}|}  -  \est X_{\ell,j} \right|^2 \\
& = \sum_{ (\ell, j ) \in E} 4 |\est Y_{\ell,j}|^2 \, \left| \frac{ \tx Y_{\ell,j}}{|\est Y_{\ell,j}|}  - \frac{\est Y_{\ell,j}}{|\est Y_{\ell,j}|}\right|^2 \\
& = \sum_{ (\ell, j ) \in E} 4 \left| \tx Y_{\ell,j} - \hat{Y}_{\ell,j} \right|^2  = 4 \Norm{\tx Y - \hat{Y}}{F}^2.\\
\end{align*}
\end{proof}

\section{Numerical evaluation}
\label{sec: numerics}

In this section we present a numerical comparison of the error bounds discussed above. Our goal is to illustrate that Theorem \ref{thm: weighted new} indeed provides superior recovery guarantees for an important class of weighted angular synchronization problems, namely those appearing in the context of ptychography, as covered by Corollaries~\ref{cor: weights normal new} and \ref{cor: squared weights}. In particular, we work with  the edge set $E_\delta$ as in \eqref{eq: E delta}, for some parameter $\delta \in [ \lfloor (d+1)/2 \rfloor ] $, which determines the neighborhood of indices for which the pairwise phase differences are known.

In our numerical experiments, we consider measurements affected by angular noise, that is, the measurements are of the form \eqref{eq: X tilde}, i.e.
\[
\est X_{\ell,j}=
\begin{cases}
e^{i (\varphi_\ell - \varphi_j + \eta_{\ell,j}) }, & (\ell,j) \in E_\delta,\\
0, & (\ell,j) \notin E_\delta,
\end{cases}
\]
with the noise model that $\eta_{\ell,j}, (\ell,j) \in E_\delta$ are independent random variables uniformly distributed on  $[-\alpha,\alpha]$ for some parameter $\alpha > 0$ 
representing the noise level.


We consider signals $\tx y$ drawn at random with coordinates $\tx y_\ell = a_\ell + i b_\ell$, where $a_\ell$ and $b_\ell$ are independent identically distributed 
standard Gaussian random variables. We assume that the $|y_\ell|$ are known, so the phases of the $\tx y_\ell$ are our unknown ground truth entries $\tx x_\ell = e^{i\varphi_\ell}$. 
In most of the following examples, we fix the dimension to be $d = 64$ and the parameter $\delta =16$, so that approximately half of the pairwise phase differences are known. 
For each point in the figures we generated 30 test signals and plot the average norm of the error. All experiments were performed on the laptop computer running Windows 10 Pro with an Intel(R) Core(TM) i7-8550U processor, with 16 GB RAM and Matlab R2018b. 

\par  

We begin with the comparison of the recovery guarantees for the different weight matrices covered by Corollaries \ref{cor: eig unweighted}, \ref{cor: weights normal new}, and \ref{cor: squared weights} in terms of the angular noise level $\alpha$ measured in degrees. To put the bounds into perspective, we include the empirical error of both SDP and ER.

Due to the fact that the coordinates of $\tilde{x}$ and $\tx x$ have modulus $1$, a na\"ive bound for the phase error is given by
\begin{equation}\label{eq: ph error bound} 
\Norm{\tilde{x} - e^{i\theta} \tx x}{2} \le \Norm{\tilde{x}}{2} + \Norm{e^{i\theta} \tx x}{2} = 2 \sqrt d.
\end{equation}

Beyond this threshold, the error bounds provided by the statements are non-informative, which is why we indicate the threshold by a dashed black line in the plots.

\par  

\begin{figure}[b!]
\centering
\hspace{0.5cm}
\begin{subfigure}[t]{.43\linewidth}
\begin{tikzpicture}[scale=0.8]
\pgfplotstableread{figures/unweighted_data.csv}{\linea}
\begin{loglogaxis}[
xlabel={Noise in degrees},
legend columns=2,
legend style={at={(0.95,-0.25)},anchor=north east,font=\fontsize{8}{8}\selectfont},
]
\addplot[Green,thick, mark = diamond*,line width=1pt] table[x=Angle, y=SDP]{\linea};
\addplot[Red,dashed,mark=square*,line width=1pt] table[x=Angle, y=Eigen]{\linea};
\addplot[Green,dashed, mark = diamond*,line width=1pt] table[x=Angle, y=Theorem22]{\linea};
\addplot[Red,thick,mark = square*,line width=1pt] table[x=Angle, y=Corollary32]{\linea};
\addplot[Blue,thick,mark=pentagon,line width=1pt] table[x=Angle, y=Theorem23]{\linea};
\addplot[Purple,thick,mark=otimes,line width=1pt] table[x=Angle, y=Theorem21]{\linea};
\addplot[Orange,thick,mark=triangle, line width=1pt] table[x=Angle, y=Theorem24]{\linea};
\addplot[Black,dashed,line width=2pt]  table[x=Angle, y expr={16}]{\linea};
\legend{
SDP emp. error,
ER emp. error,
Theorem \ref{thm: SDP},
Corollary \ref{cor: eig unweighted},
Theorem \ref{eq: SqrtEst},
Theorem \ref{thm: weak bound},
Theorem \ref{eq:prevworkweights},
Na\"ive bound \eqref{eq: ph error bound},
}
\end{loglogaxis}
\end{tikzpicture}
\caption*{}
\end{subfigure}
\hspace{1cm}
\begin{subfigure}[t]{.43\textwidth}
\begin{tikzpicture}[scale=0.8]

\pgfplotsset{
    axis on top,
    ylabel style={align=center},
}

\begin{groupplot}[
        group style={
        group size=1 by 2,
        vertical sep=55pt,
        },
        height =4.5cm,
        width = 8cm,
    ]

\nextgroupplot[	
	xmode = log,	
    	ylabel={Rank},
     title={Average rank of the SDP solutions},
     xlabel={Angle},
    ]
    \addplot  table[x = Angle, y =Rank,col sep=space]{figures/unweighted_rank_data.csv};

\nextgroupplot[	
	xmode = log,	
    	ylabel={Norm $\Norm{z}{\infty}$}, 
     title={Average norm $\Norm{z}{\infty}$ }, 
     xlabel={Angle},
    ]
    \addplot  table[x = Angle, y =Norm,col sep=space]{figures/unweighted_norm_data.csv};  
\end{groupplot}
\end{tikzpicture}%
\caption*{}
\end{subfigure}
\caption{{\em Left:} Comparison of the recovery guarantees and true errors for unweighted angular synchronization in the context of the ptychography problem, $d = 64$, $\delta=16$.\\
{\em Right:} Rank of the SDP solution and supremum norm of the ER solution as measures for the tightness of the respective relaxation.
}
\label{fig: unweighted}
\end{figure}

Both for unweighted graphs (Figure \ref{fig: unweighted}) and weighted graphs (Figures \ref{fig: amplitudes} and \ref{fig: squared amplitudes}), we observe that the empirical error performs similarly for both SDP and ER; there is no significant difference between the two methods in terms of the phase error. 
For the low and medium noise levels, the phase error rises linearly with the angular noise level. Only for very high noise it exhibits faster growth.
\par

For unweighted graphs (Figure \ref{fig: unweighted}), the guarantees provided by Corollary \ref{cor: eig unweighted} for ER more or less agree with the bounds given by Theorem \ref{thm: SDP} for the least squares problem. This is remarkable not only because ER is faster than SDP (see Figure \ref{fig: dim_vs_time} below), but also in view of the significantly larger range of tightness of the relaxation, also depicted in Figure \ref{fig: unweighted}. Namely, Corollary \ref{cor: eig unweighted} exhibits additional dimensional scaling factors only when the supremum norm of $z$ is very large, which is not the case even for a large noise level. On the other hand, the SDP relaxation is only tight when the solution is of rank one, which already fails for a noise level of a few degrees.
Thus for larger noise, the error bounds for LSP no longer apply for the SDP solution.



\begin{figure}[t!]
\centering
\hspace{0.5cm}
\begin{subfigure}[t]{.43\linewidth}
\begin{tikzpicture}[scale=0.8]
\pgfplotstableread{figures/amplitude_data.csv}{\linea}
\begin{loglogaxis}[
xlabel={Noise in degrees},
legend columns=2,
legend style={at={(0.95,-0.25)},anchor=north east,font=\fontsize{8}{8}\selectfont},
]
\addplot[Green,thick, mark = diamond*,line width=1pt] table[x=Angle, y=SDP]{\linea};
\addplot[Red,dashed,mark=square* ,line width=1pt] table[x=Angle, y=Eigen]{\linea};
\addplot[Green,dashed, mark = diamond*,line width=1pt] table[x=Angle, y=Corollary33]{\linea};
\addplot[Red,thick,mark=square* ,line width=1pt] table[x=Angle, y=Corollary33Eigen]{\linea};
\addplot[Blue,mark=pentagon,thick,line width=1pt] table[x=Angle, y=Theorem23]{\linea};
\addplot[Orange,thick,mark=triangle,line width=1pt] table[x=Angle, y=Theorem24]{\linea};
\addplot[Black ,dashed,line width=2pt]  table[x=Angle, y expr={16}]{\linea};
\legend{
SDP emp. error,
ER emp. error,
Corollary \ref{cor: weights normal new} (LSP),
Corollary \ref{cor: weights normal new} (ER),
Theorem \ref{eq: SqrtEst},
Theorem \ref{eq:prevworkweights},
Na\"ive bound \eqref{eq: ph error bound},
}
\end{loglogaxis}
\end{tikzpicture}
\caption*{}
\end{subfigure}
\hspace{1cm}
\begin{subfigure}[t]{.43\textwidth}
\begin{tikzpicture}[scale=0.8]

\pgfplotsset{
    axis on top,
    ylabel style={align=center},
}

\begin{groupplot}[
        group style={
        group size=1 by 2,
        vertical sep=55pt,
        },
        height =4.5cm,
        width = 8cm,
    ]

\nextgroupplot[	
	xmode = log,	
    	ylabel={Rank},
     title={Average rank of the SDP solutions},
     xlabel={Angle},
    ]
    \addplot  table[x = Angle, y =Rank,col sep=space]{figures/amplitude_rank_data.csv};

\nextgroupplot[	
	xmode = log,	
    	ylabel={Norm $\Norm{z}{\infty}$}, 
     title={Average norm $\Norm{z}{\infty}$ }, 
     xlabel={Angle},
    ]
    \addplot  table[x = Angle, y =Norm,col sep=space]{figures/amplitude_norm_data.csv};  
\end{groupplot}
\end{tikzpicture}%
\caption*{}
\end{subfigure}
\caption{{\em Left:} Comparison of the recovery guarantees and true errors for angular synchronization for wights as in Corollary~\ref{cor: weights normal new}, $d = 64$, $\delta=16$.\\
{\em Right:} Rank of the SDP solution and supremum norm of the ER solution as measures for the tightness of the respective relaxation.
}
\label{fig: amplitudes}
\end{figure}

\begin{figure}[b!]
\centering
\hspace{0.5cm}
\begin{subfigure}[b]{.43\linewidth}
\begin{tikzpicture}[scale=0.8]
\pgfplotstableread{figures/squared_data.csv}{\linea}
\begin{loglogaxis}[
xlabel={Noise in degrees},
legend columns=2,
legend style={at={(0.95,-0.25)},anchor=north east,font=\fontsize{8}{8}\selectfont},
]
\addplot[Green,thick, mark = diamond*,line width=1pt] table[x=Angle, y=SDP]{\linea};
\addplot[Red,dashed,mark=square*,line width=1pt] table[x=Angle, y=Eigen]{\linea};
\addplot[Green,dashed, mark = diamond*,line width=1pt] table[x=Angle, y=Corollary34]{\linea};
\addplot[Red,thick,mark=square*,line width=1pt] table[x=Angle, y=Corollary34Eigen]{\linea};
\addplot[Blue,thick,mark=pentagon,line width=1pt] table[x=Angle, y=Theorem23]{\linea};
\addplot[Orange,thick,mark=triangle,line width=1pt] table[x=Angle, y=Theorem24]{\linea};
\addplot[Black,dashed,line width=2pt]  table[x=Angle, y expr={16}]{\linea};
\legend{
SDP emp. error,
ER emp. error,
Corollary \ref{cor: squared weights} (LSP),
Corollary \ref{cor: squared weights} (ER),
Theorem \ref{eq: SqrtEst},
Theorem \ref{eq:prevworkweights},
Na\"ive bound \eqref{eq: ph error bound},
}
\end{loglogaxis}
\end{tikzpicture}
\caption*{}
\end{subfigure}
\hspace{1cm}
\begin{subfigure}[b]{.43\textwidth}
\begin{tikzpicture}[scale=0.8]

\pgfplotsset{
    axis on top,
    ylabel style={align=center},
}

\begin{groupplot}[
        group style={
        group size=1 by 2,
        vertical sep=55pt,
        },
        height =4.5cm,
        width = 8cm,
    ]

\nextgroupplot[	
	xmode = log,	
    	ylabel={Rank},
     title={Average rank of the SDP solutions},
     xlabel={Angle},
    ]
    \addplot  table[x = Angle, y =Rank,col sep=space]{figures/squared_rank_data.csv};

\nextgroupplot[	
	xmode = log,	
    	ylabel={Norm $\Norm{z}{\infty}$}, 
     title={Average norm $\Norm{z}{\infty}$ }, 
     xlabel={Angle},
    ]
    \addplot  table[x = Angle, y =Norm,col sep=space]{figures/squared_norm_data.csv};  
\end{groupplot}
\end{tikzpicture}%
\caption*{}
\end{subfigure}
\caption{{\em Left:} Comparison of the recovery guarantees and true errors for angular synchronization for wights as in Corollary~\ref{cor: squared weights}, $d = 64$, $\delta=16$.\\
{\em Right:} Rank of the SDP solution and supremum norm of the ER solution as measures for the tightness of the respective relaxation.
}
\label{fig: squared amplitudes}
\end{figure}


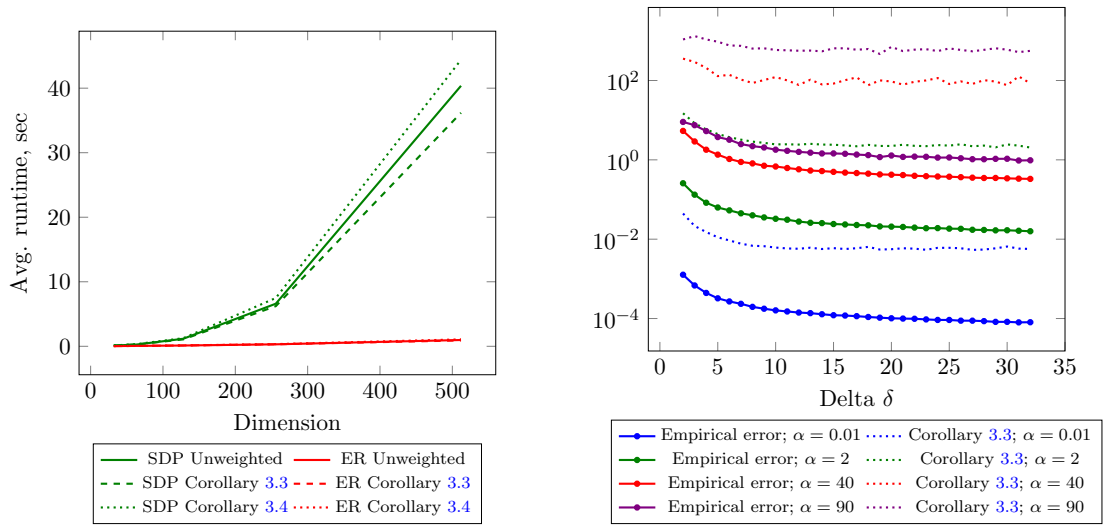
\begin{figure}[t!]
\centering
\hspace{0.5cm}
\begin{subfigure}[t]{.43\linewidth}
\begin{tikzpicture}[scale=0.8]
\pgfplotstableread{figures/time_data.csv}{\linea}
\begin{axis}[
xlabel={Dimension},
ylabel={Avg. runtime, sec},
legend style={at={(0.5,-0.2)},anchor=north ,font=\fontsize{8}{8}\selectfont},
legend columns =2
]
\addplot[Green,thick,line width=1pt] table[x=Dimension, y=SUnweighted]{\linea};
\addplot[Red,thick,line width=1pt] table[x=Dimension, y=EUnweighted]{\linea};

\addplot[Green,dashed,line width=1pt] table[x=Dimension, 
y=SAmplitudes]{\linea};
\addplot[Red,dashed,line width=1pt] table[x=Dimension, y=EAmplitudes]{\linea};

\addplot[Green,dotted,line width=1pt] table[x=Dimension, y=SSquared]{\linea};
\addplot[Red,dotted,line width=1pt] table[x=Dimension, y=ESquared]{\linea};
\legend{
SDP Unweighted,
ER Unweighted,
SDP Corollary \ref{cor: weights normal new},
ER Corollary \ref{cor: weights normal new},
SDP Corollary \ref{cor: squared weights},
ER Corollary \ref{cor: squared weights},
}
\end{axis}
\end{tikzpicture}
\caption{Time complexity of SDP and ER for weighted and unweighted graphs.}
\label{fig: dim_vs_time}
\end{subfigure}
\hspace{1cm}
\begin{subfigure}[t]{.43\textwidth}
\begin{tikzpicture}[scale=0.8]
\pgfplotstableread{figures/delta_data.csv}{\linea}
\begin{semilogyaxis}[
xlabel={Delta $\delta$},
mark size=1.0pt,
legend style={at={(0.5,-0.2)},anchor=north,font=\fontsize{8}{8}\selectfont},
legend columns=2
]
\addplot[Blue,thick,line width=1pt,mark=*] table[x=delta, y=angle001]{\linea};
\addplot[Blue,dotted,line width=1pt] table[x=delta, y=angle001b]{\linea};

\addplot[Green,thick,line width=1pt,mark=*] table[x=delta, y=angle2]{\linea};
\addplot[Green,dotted,line width=1pt] table[x=delta, y=angle2b]{\linea};

\addplot[Red,thick,line width=1pt,mark=*] table[x=delta, y=angle40]{\linea};
\addplot[Red,dotted,line width=1pt] table[x=delta, y=angle40b]{\linea};

\addplot[Purple,thick,line width=1pt,mark=*] table[x=delta, y=angle90]{\linea};
\addplot[Purple,dotted,line width=1pt] table[x=delta, y=angle90b]{\linea};
\legend{
Empirical error; $\alpha=0.01$,
Corollary \ref{cor: weights normal new}; $\alpha=0.01$,
Empirical error; $\alpha=2$,
Corollary \ref{cor: weights normal new}; $\alpha=2$,
Empirical error; $\alpha=40$,
Corollary \ref{cor: weights normal new}; $\alpha=40$,
Empirical error; $\alpha=90$,
Corollary \ref{cor: weights normal new}; $\alpha=90$,
}
\end{semilogyaxis}
\end{tikzpicture}
\caption{Impact of number of known pairwise differences (encoded by parameter $\delta$) on the empirical error and error bound for ER applied to ptychography setting (Corollary \ref{cor: weights normal new}).}
\label{fig: delta}
\end{subfigure}
\caption{Performance of the different methods and setups.}
\end{figure}

For weighted graphs 
(Figures \ref{fig: amplitudes} and \ref{fig: squared amplitudes}), our bounds for the LSP improve upon previous works for a very large range of noise levels, and again, our guarantees for ER are very close to those of LSP. This is even more relevant than in the unweighted case, as the SDP relaxation is not tight even when the noise level is as small as $10^{-3}$. 
The tightness of ER is also inferior to the unweighted case, but we do observe approximate tightness for noise levels of a few degrees.

In terms of the runtime complexity (Figure \ref{fig: dim_vs_time}), ER exhibits almost linear scaling in the dimension of the problem and clearly outperforms SDP, whose runtime exhibits quadratic scaling. This difference is to be expected as SDP lifts the problem to a $d \times d$-dimensional matrix space and thus estimates $d^2$ unknowns instead of $d$ in the case of ER. In fact the large runtime complexity is a crucial bottleneck for SDP relaxations in ptychography, where the dimensions are commonly high.




The last simulation (Figure \ref{fig: delta}) illustrates how both the empirical error and our error bounds depend on the size of the mask in ptychography (which in turn is related to the connectivity of the graph). Again, up to constants, we observe a similar decay pattern for the error between theory and experiment with a fast decay for small $\delta$ (up to $\mathcal{O}(\log d)$) and slower decay for larger values of $\delta$.






\section{Discussion and future work}

\begin{figure}[t!]
\centering
\begin{tikzpicture}[scale=0.8]

\pgfplotsset{
    axis on top,
    ylabel style={align=center},
}

\begin{groupplot}[
        group style={
        group size=2 by 1,
        horizontal sep=55pt,
        }
    ]

\nextgroupplot[	
    	ylabel={Average phase error},
     xlabel={Angle},
    ]
    \addplot[Green,thick,line width=1pt]  table[x = Angle, y =SUnweighted,col sep=space]{figures/mandw1_data.csv};
    \addplot[Green,dashed,line width=1pt]  table[x = Angle, y =SAmplitudes,col sep=space]{figures/mandw1_data.csv};
    \addplot[Green,dotted,line width=1pt]  table[x = Angle, y =SSquared,col sep=space]{figures/mandw1_data.csv};
    \addplot[Red,thick,line width=1pt]  table[x = Angle, y =EUnweighted,col sep=space]{figures/mandw1_data.csv};
    \addplot[Red,dashed,line width=1pt]  table[x = Angle, y =EAmplitudes,col sep=space]{figures/mandw1_data.csv};
    \addplot[Red,dotted,line width=1pt]  table[x = Angle, y =ESquared,col sep=space]{figures/mandw1_data.csv};

    \coordinate (c1) at (rel axis cs:0,1);

\nextgroupplot[	
	ymode = log,	
    	ylabel={Average phase error},
     xlabel={Signal to Noise Ratio},
     legend style={at={($(0,0)$)},legend columns=3,fill=none,draw=black,anchor=center,align=center},
            legend to name= leg1
    ]
    \addplot[Green,thick,line width=1pt]  table[x = SNR, y =SUnweighted,col sep=space]{figures/mandw2_data.csv};
    \addplot[Green,dashed,line width=1pt]  table[x = SNR, y =SAmplitudes,col sep=space]{figures/mandw2_data.csv};
    \addplot[Green,dotted,line width=1pt]  table[x = SNR, y =SSquared,col sep=space]{figures/mandw2_data.csv};
    \addplot[Red,thick,line width=1pt]  table[x = SNR, y =EUnweighted,col sep=space]{figures/mandw2_data.csv};
    \addplot[Red,dashed,line width=1pt]  table[x = SNR, y =EAmplitudes,col sep=space]{figures/mandw2_data.csv};
    \addplot[Red,dotted,line width=1pt]  table[x = SNR, y =ESquared,col sep=space]{figures/mandw2_data.csv};
    \legend{
SDP Unweighted,
SDP Amplitudes,
SDP Squared amplitudes,
ER Unweighted,
ER Amplitudes,
ER Squared amplitudes,
}
    \coordinate (c2) at (rel axis cs:-0.1,0);
\end{groupplot}

\node[below] at (c2 |- current bounding box.south)
{\pgfplotslegendfromname{leg1}};
\end{tikzpicture}%
\caption{Performance of the different relaxations for all weight setups. \\
{\em Left:} Random angular noise model used in the Section \ref{sec: numerics}. \\
{\em Right:} Angular synchronization as a part of the ptychographic reconstruction in \cite{Iwen.2020}\\ Noise decreases as SNR increases.}
\label{fig: methods}
\end{figure}

The main focus of this paper is the eigenvector relaxation of the angular synchronization problem. We derived new flexible error bounds for this method. Along the way, we established new recovery guarantees for the solution of the weighted least squares problem \eqref{eq: ls formulation}. Our numerical evaluation shows that obtained recovery guarantees outperform other results in the literature. As compared to semidefinite programming, eigenvector relaxation shows similar performance in the terms of empirical error and has significantly shorter runtime; at the same time our recovery guarantees are not subject to additional constraint corresponding to the tightness of the relaxation as they appear for the semidefinite programming. 

Our numerical experiments are based on the simple random angular noise model, which likely does not correspond to the noise arising in applications. Also minimizing the phase error, as explored in Section~\ref{sec: numerics} is likely not the optimal objective for ptychography. Rather a weighted error model should be used to capture the varying impact of errors arising for coefficients of different magnitude. 
We consider it to be an interesting direction for future work to explore more application driven noise models and error metrics.

A first numerical exploration in this direction is shown in Figure~\ref{fig: methods}. We observe that while for the simplified noise model, unweighted angular synchronization seems most appropriate, for Gaussian noise applied directly to the phaseless measurements \cite{Iwen.2020}, weighted angular synchronization is the method of choice.


Another interesting direction of future work is to extend the generalized power method developed by Boumal in \cite{Boumal.2016} to arbitrary sets $E$. Our current results can be considered as a first step, 
since the generalized power method uses the solution of the eigenvector relaxation problem as initialization, so good bounds on the error are crucial for estimating the quality of the initialization.

\section{Acknowledgments}
OM and FF were partially supported by the Helmholtz Association within the project Ptychography4.0.
FK acknowledges support by the German Science Foundation DFG in the context of an Emmy Noether junior research group (project KR 4512/1-1).

This is a preprint of an article published in Journal of Fourier Analysis and Applications. The final authenticated version is available online at:\\ 
\url{https://doi.org/10.1007/s00041-021-09834-1}

\bibliography{} 
\end{document}